\documentclass[preprint,10pt]{elsarticle}
\usepackage{amsthm,amsmath,amssymb,amsfonts}
\usepackage[american]{babel}
\usepackage{url}
\usepackage{indentfirst}
\usepackage[latin1]{inputenc}
\usepackage{geometry}
\usepackage{amsfonts}

\allowdisplaybreaks

\newtheorem{theorem}{Theorem}[section] 
\newtheorem*{theorem*}{Theorem}
\newtheorem{lemma}[theorem]{Lemma}
\newtheorem{proposition}[theorem]{Proposition}
\newtheorem{corollary}[theorem]{Corollary}

\theoremstyle{definition}
\newtheorem{definition}[theorem]{Definition}

\theoremstyle{remark}
\newtheorem{remark}[theorem]{Remark}

\newcommand{\Z}{{\mathbb{Z}}}   

\newcommand{\K}{{\mathbb{K}}}   

\newcommand{\Q}{{\mathbb{Q}}}   

\newcommand{\U}{{\mathcal{U}}}   
\newcommand{\oo}{{\mathfrak{o}}}

\newcommand{\A}{{\mathcal{A}}}

\newcommand{\h}{{\mathbf{H}}}

\newcommand{\n}{{\eta}}
\newcommand{\z}{{\mathcal{Z}}}  

\journal{ be a}

\begin{document}

\begin{frontmatter}


\title{Alternative algebras with the hyperbolic property}

\author[a]{Juriaans, S. O.  \footnote{The author is supported by CNPq and FAPESP-Brazil}}
\author[a]{Polcino Milies, C. \footnote{The author is supported by CNPq-Brazil}}
\author[b]{Souza Filho, A.C. \footnote{The author is supported  by FAPESP-Brazil}
\footnote{This work is part of the third author's Ph.D thesis (\cite{thss}) which was written under the supervision of the first author.}}
\address[a]{Universidade de S\~ao Paulo, Instituto de Matem\'atica e Estat\'\i stica (IME-USP)
\\Caixa Postal 66281, S\~ao Paulo, CEP
05315-970 - Brasil\\ E-mails:  ostanley@usp.br and  polcino@ime.usp.br}
\address[b]{Universidade de S\~ao Paulo, Escola de Artes, Ci\^encias e Humanidades (EACH-USP),
\\Rua Arlindo B\'ettio, 1000, Ermelindo Matarazzo,
S\~ao Paulo,CEP
 03828-000 - Brasil\\ E-mail: acsouzafilho@usp.br}

\begin{abstract} We investigate the 
 structure of an alternative finite dimensional $\Q$-algebra $\mathfrak{A}$ subject to the condition that for a   $\Z$-order $\Gamma \subset \mathfrak{A}$, and thus for every $\Z$-order of $\mathfrak{A}$,  the loop of units of $\U (\Gamma )$ does not contain 
a free abelian subgroup of rank two. In particular, we prove that the radical  of such an algebra associates with the whole algebra.  We also classify $RA$-loops $L$ for which $\mathbb{Z}L$ has this property.  The classification for group rings is still an open problem.
\end{abstract}

\begin{keyword}


Alternative Algebras  \sep RA-Loops \sep Unit \sep Hyperbolic Group \sep Orders

\end{keyword}

\end{frontmatter} 

{\scriptsize {$2000$ Mathematics Subject Classification. Primary 20N05, 17D05. Secondary 16S34}}

\section{Introduction}

Group rings $\Z G$ whose unit groups $\U(\Z G)$ are hyperbolic
were characterized in \cite{jpp} in the case when $G$ is polycyclic-by-finite. A similar
question was considered  for $RG$, $R$ being the ring of 
algebraic integers of  $\mathbb{K}=\mathbb{Q}(\sqrt{-d})$ and $G$ a finite  group (see \cite{jps}).   In \cite{ijs,ijsf}, these results were extended to characterize  associative algebras $\A$ finite dimensional over the rational numbers containing  a  $\Z$-order $\Gamma\subset \A$ whose unit group $\U(\Gamma)$ does not contain a subgroup isomorphic to a free abelian group of rank two. An algebra $\A$ with this property is said to have the {\it hyperbolic property}. Using these general results, the finite semigroups $S$ and the fields $\mathbb{K}=\mathbb{Q}(\sqrt{-d})$ such that $\mathbb{K} S$ has the hyperbolic property were classified.

In this paper we study the same problem in the context of non-associative algebras, in particular those which are loop algebras. A loop $L$ is a nonempty set with a closed binary
operation $\cdot$ relative to which there is a two-sided identity and such that the right and left translation maps $R_x(g):=g
\cdot x$ and $L_x(g):=x \cdot g$ are bijections. Let $L$ be a finitely generated loop, we say that $L$ has the hyperbolic property if it does not contain a free
abelian subgroup of rank two. This definition is an extension of the notion of hyperbolic group
defined by Gromov \cite{grmv} via the Flat Plane Theorem
\cite[Corollary $III.\Gamma.3.10.(2)$]{mbah}. 
Here we characterize the $RA$-loops $L$  such that $\U(\Z L)$ has the hyperbolic property. 


In section $2$, we fix  notation and give definitions. In section $3$, we prove a structure theorem for finite dimensional alternative algebras with the hyperbolic property. As a consequence we obtain that, for these algebras, the radical associates with the whole algebra. In the last section we present a full classification
of those $RA$-loops $L$  whose  unit loop $\U(\Z L)$ has the hyperbolic property. We remark that this problem is not yet completely settled for groups.


\section{The Hyperbolic Property}

For a field $\mathbb{K}$ let $\h(\K)=(\frac{\alpha,\beta}{\K}), \ \alpha,\beta \in \K$ be the  
generalized quaternion algebra over $\mathbb{\K}$, i.e.,
 $\h(\K)=\K[i,j:i^{2}=\alpha,j^{2}=\beta,ij=-ji=k].$ The set $\{1,i,j,k\}$ is
a $\K$-basis of $\h(\K)$. Recall that an algebraic number $a$ is totally real if all its conjugates are real. A field $\K$ is totally real if every element $a \in \K$ is totally real. A number $a \in \K$ is called totally positive if all its conjugates are real and positive. If $\K$ is totally real and $\alpha,\beta$ are
totally positive, the algebra $\h(\K)$  is said to be a totally definite
quaternion algebra. 
The map $\n:\h(\K) \longrightarrow \K $,  $\n (x=x_{1}+x_{i}i+x_{j}j+x_{k}k) = x_{1}^{2}-\alpha x_{i}^{2}-\beta x_{j}^{2}+\alpha\beta x_{k}^{2}$ is called the  norm of $\h(\K)$.

Denoting by $[x,y,z]\dot{=}(xy)z-x(yz)$, we recall that a ring $A$ is
alternative if $[x,x,y]=[y,x,x]=0$, for every $x,y \in A$. 
Let $\mathfrak{A}$ be a finite dimension alternative $\mathbb{Q}$-algebra and $\mathfrak{R}$ the radical of $\mathfrak{A}$. Recall that the radical of a finite dimensional algebra is its unique maximal nilideal. In \cite[chapter $III$]{sch}, it is proved that $\mathfrak{R}$ is the set of the elements $z \in \mathfrak{A}$, such that, $az$ is nilpotent for all $a \in \mathfrak{A}$. Elements $z$ with this property are called properly nilpotent. 
 By \cite[Theorem $3.18$]{sch}, $\mathfrak{A} \cong
\mathfrak{S} \oplus \mathfrak{R}$ a direct sum as vector spaces, where $\mathfrak{S}$ is a
subalgebra of $\mathfrak{A}$ and $\mathfrak{S}\cong
\mathfrak{A}/\mathfrak{R}$ is semi-simple.

\begin{definition} Let $\K$ be a field, of characteristic zero, and let $\mathfrak{A}$ be a finite dimensional alternative $\K$-algebra. We say $\mathfrak{A}$ has
the {\it hyperbolic property} if there exists a $\Z$-order $\Gamma \subset
\mathfrak{A}$ whose  unit loop $\U(\Gamma)$ has the hyperbolic property.
\end{definition}

For an associative finite dimensional $\Q$-algebra this property was defined in \cite{ijs}. 

\begin{proposition}
\label{raddim} 

Let $\mathfrak{A}$ be an alternative finite dimension
$\Q$-algebra such that $\mathfrak{A} \cong \mathfrak{S} \oplus
\mathfrak{R}$, with $\mathfrak{R} \neq 0$ being the radical of $\mathfrak{A}$. If $\mathfrak{A}$ has the hyperbolic property, then $\mathfrak{R}$ is nilpotent of index $2$. Furthermore, there exists $j_0 \in
\mathfrak{R}$  such that and $\mathfrak{R}\cong \langle j_0
\rangle_{\Q}$ is the $\Q$-linear span of $j_0$ over
$\Q$.

\end{proposition}

\begin{proof} 
As $\mathfrak{R}$ is nilpotent, there exists a positive integer $n$ and the chain of ideals $\mathfrak{R} \supset \mathfrak{R}^2 \cdots \supset
  \mathfrak{R}^n=0$. Suppose, $n \geq 3$.
  Set $x,y \in \mathfrak{A}$, such that, $x \in \mathfrak{R}^{n-2} \setminus
 \mathfrak{R}^{n-1}$ and $y \in \mathfrak{R}^{n-1}$. For a $\Z$-order $\Gamma \subset \mathfrak{A}$, there exist
  $\alpha, \beta \in \Z$, such that, $\{\alpha x, \beta y \} \subset (\Gamma \cap \mathfrak{R})$. Clearly,
   $y^2=0=x^3=0$ and $\alpha x, \beta y$ are nilpotent elements. Hence $(1+\alpha x), (1+\beta y) \in \U(\Gamma)$. We claim that $\langle 1+\alpha x \rangle \cap \langle
     1+\beta y)\rangle =\{1\}$. In fact, since by Artin's Theorem, \cite{sch}, the subalgebra generated by $x,y$ is an associative algebra, for a positive integer $r$, we have that 
 $(1+\alpha x)^r=1+r \alpha x+\frac{r^2-r}{2}(\alpha x)^2$. For a negative exponent, write $(1+\alpha x)^{-r}=((1+\alpha x)^{-1})^r=(1-\alpha x +(\alpha x)^2)^r=1-r(\alpha x-(\alpha x)^2)+{ r \choose 2}(\alpha x)^2=1-r (\alpha x)+\frac{r^2+r}{2} (\alpha x)^2$. So, for an arbitrary $r \in \Z$, we can write $(1+\alpha x)^r=1+|r| \alpha x+\frac{r^2-|r|}{2}(\alpha x)^2$, where $|r|$ is the absolute value of the integer $r$.  Similarly, for $s \in \Z$ we have $(1+\beta y)^s=1+|s| \beta y$. So, assume that, for some $r,s \in \Z$, we have $(1+\alpha x)^r= (1+\beta y)^s$. Hence, $1+|r| \alpha x+\frac{r^2-|r|}{2}(\alpha x)^2=1+|s| \beta y$, and thus, if $r \neq 0$, we can write $x=\frac{1}{|r|\alpha}(1+|s| \beta y-\frac{r^2-|r|}{2}(\alpha x)^2) \in \mathfrak{R}^{n-1}$, a contradiction.
 As $xy=yx$, then we would have that $\langle 1+\alpha x, 1+\beta y\rangle \cong \Z^2$. 
 Since $\mathfrak{A}$ has the
       hyperbolic property, then we must have that $n \leq 2$. 
       
       As $\mathfrak{R} \neq 0$, then $n=2$. Now we prove that $dim(\mathfrak{R})_{\Q}=1$. Suppose $dim(\mathfrak{R})_{\Q}>1$, then there exist two linearly independent elements $x,y \in \mathfrak{R}$, and we compute $(1+\alpha x)^r=1+|r| \alpha x$ and $(1+\beta y)^s=1+|s| \beta y$. So, as before, we have that $\langle 1+\alpha x \rangle \cap \langle
     1+\beta y)\rangle =\{1\}$, otherwise we would get $x= \frac{|s| \beta}{|r| \alpha}y$. This can not happen, as 
       $\mathfrak{A}$ has the hyperbolic property. Thus, $dim(\mathfrak{R})_{\Q}=1$, and there exists $j_0$ as in the proposition.
\end{proof}

\begin{definition}[Totally Definite Octonion Algebra] \label{tdoa} Let $\K$ be a field.
An alternative division algebra $\mathfrak{A}$, with center $\K$, is called an {\it Octonion Algebra} if there exist $x,y,z \in \mathfrak{A}$, such that, the set $\mathcal{B}=\{1,x,y,xy\} \cup \{z,xz,yz,(xy)z\}$ is a $\K$-basis of
$\mathfrak{A}$, with $x^2=-\alpha$, $y^2=-\beta$,
$z^2=-\gamma$ and $\alpha, \beta,\gamma \in \K$. If $\K$ is totally real, and the elements $\alpha, \beta, \gamma$ are totally positive,  then we say that $\mathfrak{A}$ is a {\it totally definite
octonion algebra}. In this case we write $\mathfrak{A} =(\K, -\alpha,-\beta,-\gamma)$.
\end{definition}

The totally definite octonion algebra
$\mathfrak{A}:=(\K,-\alpha,-\beta,-\gamma)$ is non-split: since
$\alpha, \beta,\gamma$ are totally positive and $\K$ is totally real,
the equation $x^2+\alpha y^2+\beta z^2+\alpha\beta w^2+\gamma t^2=0$, $x,y,z \in \K$,
has only the trivial solution. Thus, by \cite[Theorem 3.4]{gjp}, $\mathfrak{A}$ is non-split.

Next we give a characterization of the totally definite octonion algebras which is a natural extension 
of \cite[Lemma $21.3$]{uni}. Recall that by a corolary in \cite[p. $152$]{shest}, since $\mathfrak{A}$ is an alternative algebra which is a division ring that is not associative, its center is a field and $\mathfrak{A}$ is a Cayley-Dickson algebra over its center.

Let $\Gamma_1, \Gamma_2$ be $\Z$-orders of  an alternativa algebra $\mathfrak{A}$ over its center $\K$ a number field. Some properties, for orders of associative algebras, still hold in the non-associative case. For example, $\Gamma_1 \cap \Gamma_2$ is an order of $\mathfrak{A}$. Consequently, if $\mathfrak{O}$ is a maximal order in $\mathfrak{A}$ and $\oo_\K$ the ring of integers of $\K$ which is a $\Z$-order of $\K$, then $(\mathfrak{O}\cap \K) \cap \oo_\K$ is a $\Z$-order of $\K$. It is well known, that $\oo_\K$ is the  unique maximal $\Z$-order of $\K$. Hence, $(\mathfrak{O}\cap \K) \cap \oo_\K = \oo_\K$, so $\oo_\K \subseteq (\mathfrak{O} \cap \K)$ and $(\mathfrak{O} \cap \K)=\oo_\K$.

\begin{lemma}
Let $\A$ be a finite dimensional algebra over $\K$, $B=\{b_1, \cdots, b_n\}$ a basis of $\A$ over $\K$ and $\oo_{\K}$ the ring of algebraic integers of $\K$. Set $O_{\A}=\oo_\K+\oo_\K b_1+ \cdots + \oo_\K b_n$. Then, every maximal $\Z$-order $\Lambda$ of $\A$ contains the order $O_{\A}$.  
\end{lemma}

\begin{proof}
Since $\Lambda$ is a $\Z$-order, it is finitely generated as $\Z$-module. Also $\Lambda$ contais a $\Q$-basis of $\A$. Clearly, $R=O_{\A}+\Lambda$ is finitely generated as $\Z$-module. Thus $R$ is a $\Z$-order of $\A$ which contains $\Lambda$. The maximality of $\Lambda$ yields that  $O_{\A} \subset \Lambda$.
\end{proof}

\begin{remark} 
Given any $\Z$-order $O$ in $\A$, it is contained in a maximal $\Z$-order $\Lambda$. As both $[\U(\Lambda):\U(O)]$ and $[\U(\Lambda):\U(O_\A)]$ are finite, it follows that any unit loop of the form $\U(O)$ is finite if and only if $\U(O_A)$ is finite.
\end{remark}

\begin{theorem}
\label{uniex}

Let  $\mathfrak{A}$ be a non commutative and nonassociative alternative division algebra, finite 
dimensional over its center $\K$. Suppose that $\K$ is a number field,  $\oo_{\K} $ its ring of
algebraic integers and  $\mathfrak{O}$  a maximal order in
$\mathfrak{A}$. Then, the index $|\U(\mathfrak{O}):\U(\oo_\K)|<\infty$ if and only if the algebra $\mathfrak{A}$ is a totally definite octonion algebra.
\end{theorem}

\begin{proof} 
Assume that $|\U(\mathfrak{O}):\U(\oo_\K)|<\infty$. As we observed above, $\mathfrak{A}$ is a Cayley-Dickson algebra $(\K,-\alpha,-\beta,-\gamma)$, where $\alpha, \beta, \gamma \in \K$. Thus, $\mathfrak{A}$ is a composition algebra $(\A, -\gamma)$, where $\A$ is a quaternion algebra $H(\K)=(\frac{-\alpha, -\beta}{\K})$, with $\alpha, \beta \in \K$. Clearly, we can suppose $\alpha, \beta \in \mathfrak{O} \cap \K$, so that, $x^2=-\alpha$, $y^2=-\beta$ and $O_\A= \oo_\K+\oo_\K x+\oo_\K y+\oo_\K xy$ is a $\Z$-order of $\A$. By the previous lemma, $O_\A \subset \mathfrak{O}$, so $\oo_\K \subset O_\A \subset \mathfrak{O}$. Since $[\U(\mathfrak{O}):\U(\oo_\K)]$ is finite, then $[\U(O_\A):\U(\oo_\K)]$ is finite. By \cite[Lemma 21.3]{uni}, $\A$ is a totally definite quaternion algebra, $\K$ is totally real and $\alpha, \beta$ are totally positive. Similarly, we can interchange $\gamma$ with $\alpha$ or $\beta$ and conclude also that $\gamma$ is totally positive. Thus $\mathfrak{A}$ is a totally definite octonion algebra.

To prove the converse, we follow the arguments of \cite[Lemma $21.3$]{uni}. Let $B=\{1,x,y,z,xy,xz,$ $yz,(xy)z\}$ be a basis of $\mathfrak{A}$ over a totally real field $\K$, with $x^2=-\alpha$, $y^2=-\beta$, $z^2=-\gamma$ and $\alpha, \beta,\gamma \in \K$. For a set $X \subset \mathfrak{A}$, we shall denote by $L_1(X)$ the set of units of $X$ with  norm equal to $1$. 

Without loss of generality, we may take $\alpha, \beta, \gamma \in \oo_\K$ and define the  $\Z$-order $O=\oplus_{b \in B} \oo_{\K} b$. By the lemma above, $O \subset \mathfrak{O}$. By \cite[Lemma $VIII.2.6$]{gjp}, the index $[\U(\mathfrak{O}):\U(O)]$ is finite.

We claim that $L_1(O)$ is finite. To see this, notice that if $u \in L_1(O)$, then $u= \sum_{b \in B}c_i b$, with $c_i \in \oo_\K$, and $n(u)= c_1^2+\alpha c_2^2+\beta c_3^2+\gamma c_4^2 \alpha \beta+c_5^2 \alpha \gamma+c_6^2 \beta \gamma +c_7^2 \alpha \beta \gamma=1$. Since $\alpha, \beta, \gamma$ are totally positive elements, it follows that for all embeddings $\sigma$ of $\K$, the absolute values $|c^{\sigma}_i|$ are bounded for all $c_i$, $1\leq i \leq 7$. Thus $L_1(O)$ is finite.

 Since the norm $n(a)=a \overline{a}\in \K$ for all $a \in \mathfrak{A}$, consider the map $\varphi:\U(\mathfrak{O})/\U(\oo_\K)
\longrightarrow \U(\oo_\K)/(\U(\oo_\K))^2$, $\varphi(x \U(\oo_\K))=n(x) (\U(\oo_\K))^2$, which is easily seen to be well defined. If $x\U(\oo_\K) \in ker(\varphi)$,
then $\varphi(x \U(\oo_\K))=n(x) (\U(\oo_\K))^2 \in (\U(\oo_\K))^2$ and
$ n(x)=\lambda^2 \in (\U(\oo_\K))^2 \cap \K$. Define $z=\lambda^{-1}x \in \U(\mathfrak{O})$. Clearly $n(z)=1$  thus $\lambda^{-1}x \in L_1(\mathfrak{O})$
and the kernel of $\varphi$ is
$L_1(\mathfrak{O})\U(\oo_\K)/\U(\oo_\K)$. So $\frac{\U(\mathfrak{O})}{L_1(\mathfrak{O}) \U(\oo_\K)} \cong \frac{\U(\oo_\K)}{(\U(\oo_\K))^2}$ which is finite. As $L_1(\mathfrak{O})$ is finite, we get that $|\U(\mathfrak{O})/\U(\oo_\K)|$ is also finite, as desired.
\end{proof}

\begin{corollary}\label{cor} Let $\mathfrak{A}=\K(-\alpha,-\beta,-\gamma)$ be a totally definite octonion
algebra over a number field $\K$, and let $\mathfrak{O} \subset
\mathfrak{A}$ be a maximal $\Z$-order of $\mathfrak{A}$.  The
unit loop $\U(\mathfrak{O})$ has the hyperbolic property if and only if $\K$ is either $\Q$ or 
$\Q(\sqrt{-d})$, where $d<1$ is a square free integer.
\end{corollary}

\begin{proof} 
 Suppose $\U(\mathfrak{O})$ has the hyperbolic property. Let $\oo_{\K} $ be the ring of algebraic integers of $\K$. Then $\U(\oo_{\K}) \subset \U(\mathfrak{O})$. If $\U(\mathfrak{O})$ is finite, then $\U(\oo_{\K}) $ is a
finite group.  Since $\K$ is totally real, by Dirichlet's Unit
Theorem, $\K=\Q$. Suppose $\U(\mathfrak{O})$  is infinite.  By Theorem \ref{uniex} , $|\U(\mathfrak{O}):\U(\oo_{\K})|$
is finite, therefore $\U(\oo_{\K})$ is infinite. As $\U(\mathfrak{O}) \supset \U(\oo_{\K})$ has the hyperbolic property, and $\U(\oo_{\K})$ is infinite abelian, it has free rank equal to $1$. Thus, by Dirichlet's Unit
Theorem, $\K=\Q(\sqrt{-d})$, where the integer $d<1$.

Conversely, if $\K=\Q$, then $\U(\oo_\K)$ is finite and, by the previous theorem, $\U(\mathfrak{O})$ is finite and has the hyperbolic property. If
$\K=\Q(\sqrt{-d}), d<1$, then, by Dirichlet Unit Theorem, $\U(\oo_\K)
\cong \Z$, and the last theorem shows  that 
$\Z^2$ cannot be contained in $\U(\mathfrak{O})$. Hence $\U(\mathfrak{O})$ has the hyperbolic property.
\end{proof}


\section{A Structure Theorem}

In this section we prove a structure theorem for finite dimensional alternative $\Q$-algebras with the hyperbolic property. 

For a commutative associative unitary ring $R$, we shall denote by  $\mathfrak{Z}(R)$  the Zorn vector matrix algebra  over $R$, defined as in \cite[p. 21]{gjp}

\begin{lemma}\label{pzorn}
The Zorn vector matrix algebra over $\Q$, $\mathfrak{Z} (\Q)$, does not have the hyperbolic property.
\end{lemma}

\begin{proof} The Zorn vector matrix algebra $\Lambda= \mathfrak{Z} (\Z)$ is a $\Z$-order of $\mathfrak{Z}
(\Q)$, setting $e_{1}=(1,0,0)$,
$e_{2}=(0,1,0)$ and $(0)=(0,0,0)$, then the elements
$\theta_{1}=\left ( \begin{array}{ll}
      0&e_{1} \\ (0)&0
            \end{array} \right )$ and $\theta_{2}=\left ( \begin{array}{ll}
      0&e_{2} \\ (0)&0
            \end{array} \right )$
are communting nilpotent elements of index $2$ which are $\Z$-independent. Hence, $\langle 1+\theta_1,1+\theta_2 \rangle \cong \Z \times \Z$.
\end{proof}

\begin{theorem}
Let $\mathfrak{A}$ be a finite dimensional, semisimple, alternative $\Q$-algebra. Write $$\mathfrak{A} = \oplus \A_i,$$ a finite direct sum of simple alternative algebras and, for each index $i$, let $\Gamma_i$ be a $\Z$-order of $\A_i$. Then, the following holds.
\begin{enumerate}
\item If $\mathfrak{A}$ contains no non-zero nilpotent elements, then $\mathfrak{A}$ has the hyperbolic property if and only if each simple component $\A_i$ is a division ring and, at most, for one index $i_0$, the loop $\U(\Gamma_{i_{0}})$ is infinite and with the hyperbolic property. The other simple components are isomorphic to either:
\begin{description}
\item [($i$)] The rational field $\Q$ or a quadratic imaginary extension of $\Q$.
\item [($ii$)] A totally definite quaternion algebra over $\Q$.
\item [($iii$)] A totally definite octonion algebra over $\Q$.
\end{description}
\item If $\mathfrak{A}$ contains nilpotent elements, then precisely one simple component is isomorphic to $M_2(\Q)$ and all the others are division algebras, with finite loop of units, as in $(i), (ii), (iii)$ above.  
\end{enumerate}
\end{theorem}
\begin{proof}
$(1)$ Assume that $\mathfrak{A}$ has the hyperbolic property. As both matrix algebras and Zorn's vector matrix algebras contain nilpotents elements, it follows that $\mathfrak{A}$ is a direct sum of division rings. Clearly, at most one of these summands can contain an order $\Gamma_{i_{0}}$ with $\U(\Gamma_{i_{0}})$ infinite and with the hyperbolic property, as otherwise $\mathfrak{A}$ would not have this property.

The fact that the simple components are as stated follows from \cite[Lemmas $21.2$ and $21.3$]{uni}, in the associative case, and from Theorem \ref{uniex}, otherwise. 

The converse is trivial.

$(2)$ Assume that $\mathfrak{A}$ has the hyperbolic property. Because of Lemma \ref{pzorn}, we know that the non associative components must be division algebras.

Since the sum of all non associative components also has the hyperbolic property, it follows from \cite
[Theorem $3.1$]{ijs} that it must be as stated and the result follows. 

Once again, the converse is trivial. 

Notice that in this situation, the component with infinite loop of units is the one isomorphic to $M_2(\Q)$, as $\U(GL_2(\Z))$ is hyperbolic. 
\end{proof}

\begin{theorem}\label{thrma}
Let $\mathfrak{A}$ be a finite dimensional alternative $\Q$-algebra with non trivial radical $J=J(\mathfrak{A})$. Write $$\mathfrak{A} =  \A \oplus J,$$ a direct sum of $\Q$-vector spaces, where $\A=\oplus_{i \in I}\A_i$, is a finite direct sum of simple alternative algebras and, for each index $i$, let $\Gamma_i$ be a $\Z$-order of $\A_i$.  

The algebra $\mathfrak{A}$ has the hyperbolic property if and only if either:
\begin{enumerate}
\item $J$ is central, $1$-dimension over $\Q$, and the simple components of $\A$ are as described in $(i),(ii)$ and $(iii)$ of the previous theorem;
\item $J$ is non-central, $1$-dimension over $\Q$, and there exist indexes $i_1,i_2 \in I$, such that $\A_{i_1}\oplus\A_{i_1}\oplus J= T_2(\Q)$, the ring of $2\times 2$ upper triangular matrices over $\Q$, and the simple components $\A_i$, with $i \neq i_1,i_2$, are as described in $(i),(ii)$ and $(iii)$ of the previous theorem.
\end{enumerate}
Furthermore, in both cases, we have that $[J,\mathfrak{A},\mathfrak{A}]$=$[\mathfrak{A},J,\mathfrak{A}]$=$[\mathfrak{A},\mathfrak{A},J]=0$.
\end{theorem}

\begin{proof}Assume that $\mathfrak{A}$ has the hyperbolic property and $J$ is central. As shown in \ref{raddim}, $J$ is of dimension $1$ over $\Q$; so, there exists $j_0 \in J$ such that $J=\Q j_0$. Since $J = \langle 1+j_0 \rangle \cong \Z$ is central, it follows that for any $\Z$-order $\Gamma_i \subset \A_i$, the group $\U( \Gamma_i)$ must be finite and the result follows. Conversely, if $(1)$ holds, it follows trivially that $\mathfrak{A}$ has the hyperbolic property.

Assume that $\mathfrak{A}$ has the hyperbolic property and $J$ is non central. As above, $J$ is of the form $J=\Q j_0$, for some $j_0 \in J$. Proposition \ref{raddim} shows that $j_0^2=0$. For an idempotent $e \in \mathfrak{A}$, we have that $e j_0 \in J$, so $e j_0=\lambda j_0$, $\lambda \in \Q$. Since $e^2=e$ and $\mathfrak{A}$ is diassociative, we have $e(e j_0)=e j_0=\lambda j_0$ and also $e^2=(\lambda j_0)=\lambda e j_0=\lambda^2 j_0$. Thus $\lambda^2=\lambda$, so $\lambda=0$ or $1$.

Write $1=\sum_{i \in I}e_i$, with $e_i \in \A_i$. For each index $i \in I$, write $e_i j_0=\lambda_i j_0$. Then $j_0=(\sum_{i \in I}e_i)j_0=(\sum_{i \in I}\lambda_i)j_0$, so $(\sum_{i \in I}\lambda_i)=1$. Thus, there exists a unique index $i_1 \in I$ such that $\lambda_{i_1}=1$. Consequentely, $e_{i_1}j_0=j_0$, and $e_i j_0=0$ if $i \neq i_1$.
In a similar way, multiplying on the right, we see that there exists a unique index $i_2 \in I$ such that $j_0 e_{i_2}=j_0$ and $j_0 e_i=0$ if $i \neq i_2$.

Notice that $i_1=i_2$ would imply $j_0$ central and thus, also $J$ central. So we may assume that $i_1 \neq i_2$.

Set $M_{i_1}=\{m \in \A_{i_1}\ mj_0=0\}$, the left annihilator of $j_0$ in $\A_{i_1}$. We wish to prove that $M_{i_1}$ is an ideal, even if $\A_{i_1}$ is non associative.
Choose $a \in \A_{i_1}$ and $m \in M_{i_1}$. We shall show that $am \in M_{i_1}$. For  an element $a \in \A_{i_1}$ we shall denote by $\lambda_a$, the elements in $\Q$ such that $a j_0=\lambda_a j_0$.

Recall, see \cite[Proposition $I.1.11$]{gjp}, that for any $x,y,z$ in $\mathfrak{A}$, we have that $$(xy)[x,y,z]=y(x[x,y,z]) \ \ (1).$$ Now, as $[a,m,j_0]=(am)j_0-a(mj_0)=(am) j_0=\lambda_{am} j_0$, in the formula above, we have $(am) \lambda_{am} j_0=m(a \lambda_{am} j_0)$. If $\lambda_{am}=0$, clearly $(am) j_0=0$. If we have $\lambda_{am} \neq 0$, we can cancel to get $(am) j_0=m(a j_0)=m(\lambda_ a j_0)=\lambda_a (m j_0)=0$, as desired.

Now, we claim that also $ma \in M_{i_1}$. Notice that $[m,a,j_0]=(ma)j_0-m(a j_0)=(ma) j_0- \lambda_a m j_0=(ma)j_0$.
Using again formula $(1)$, with $x=m, \ y=a$ and $z=j_0$, we get $(ma)((ma) j_0)=a(m((ma)j_0))=\lambda_{ma} a(m j_0)=0$; $(ma)((ma)j_0))=(\lambda_{ma})^2 j_0$, so $\lambda_{ma}=0$, proving that $(ma)j_0=0$. Hence $ma \in M_{i_1}$ and $M_{i_1}$ is an ideal.

Assume now that $M_{i_1} \neq 0$. As $\A_{i_1}$ is simple, it follows that $M_{i_1}=\A_{i_1}$. For any $a \in \A_{i_1}$, we have $a j_0= \lambda_a j_0$ and it follows that $0=a j_0=\lambda_a j_0$, implying that $a=0$, and thus $\A_{i_1}=0$, a contradiction. Thus we have shown that $M_{i_1}=0$. Since $a-\lambda_a e_{i_1} \in M_{i_1}=(0)$, all $a \in \A$ are of the form $a=\lambda_a e_{i_1}$. Hence $\A = \Q e_{i_1}$.

In a similar way, it can be shown that $\A_{i_2} = \Q e_{i_2}$.

By \cite[Theorem $3.6$, item $(iv))$]{ijs} we have that $\A_1\oplus \A_N \oplus J \cong  T_2(\Q)$. The fact that the other simple components are as stated follows again from Lemma \ref{sojj} and Lemma \ref{pzorn}.

Notice that since $J^2=(0)$, the arguments above also show that $M_L$, the left annihilator of $J$ in  $\mathfrak{A}$ is $M_L=\oplus_{i \neq i_1}\A_1 \oplus J$. 

Finally, choose $x,y \in \A$ and $j=\lambda j_0 \in J$. Then $[x,y,j_0]=(\lambda_{xy}-\lambda_x \lambda_y)j_0$.
Write $x=x_L \oplus x_1$ and $y=y_L \oplus y_1$, with $x_L,y_L \in M_L$ and $x_1,y_1 \in \A_{i_1}$. Then $(xy)j_0=(x_Ly_L+ x_1y_L+x_L y_1+ x_1y_1)j_0=(x_1y_1)j_0$, showing that $\lambda_{xy}=\lambda_x \lambda_y$, and thus $[x,y,j]=0$, for all $x,y \in \mathfrak{A}$.

Similarly, it follows that $[J,\mathfrak{A},\mathfrak{A}]=[\mathfrak{A},J,\mathfrak{A}]=0$, and the proof is complete.

\end{proof}


\section{Alternative loop algebras}

Recall that   a loop
 $L$ is an $RA$-loop if its loop algebra $R L$ over some commutative, associative
and unitary ring $R$ of characteristic not equal to $2$ is
alternative, but not associative (see \cite{gjp}).
 
In this section we classify the $RA$-loops $L$, such that the loop of units  $ \U(\Z L)$ has the hyperbolic property. 

$RA$-loops are  Moufang Loops. The following duplication process, due to Chein, \cite{chein},  results in Moufang loops. In particular, all $RA$-loops are obtained in this way. Let $G$ be a nonabelian group, $g_0 \in \z(G)$ be a central element,
$\star:G \rightarrow G$  an involution such that $g_0^\star=g_0$
and $gg^\star \in \z(G)$, for all $g \in G$, and $u$ be an
indeterminate. The set $L=G \dot{\cup} G u=M(G,\star,g_0)$, with
the operations
$$\begin{array}{l}
\textrm{1. } (g)(hu)=(hg)u;\\
\textrm{2. } (gu)h=(gh^{\star})u;\\
\textrm{3. }(gu)(hu)=g_{0}h^{\star}g,
\end{array}$$
is a Moufang Loop (see \cite{gjp}).

A Hamiltonian loop is a
non-associative loop $L$ whose subloops are all normal. A theorem of
Norton gives a complete characterization of these loops ( \cite[Theorem $II.8$]{gjp}). 

\begin{proposition}\cite[Theorem $VIII.3.2$]{gjp} Let $L$ be a torsion group or a torsion $RA$ loop. Then $\U(\Z L)=\pm L$ if and only if $L$ is an abelian group of exponent $1,2,3,4$ or $6$, or a Hamiltonian $2$-group, or a Hamiltonian Moufang $2$-loop.

\end{proposition}

\begin{lemma} \cite[ Theorem $2.3$]{gp} \label{pgh}
Let $G$ be a Hamiltonian $2$-group and $L=M(G, *,g_{0})$. Then $L$ is an
$RA$-loop which is a Hamiltonian $2$-loop and $\U_1(\Z
L)=L$. 
\end{lemma}

\begin{lemma} \label{sojj}
Let $L$ be a finite $RA$-loop. If the algebra $\Q L$ has nonzero
nilpotent elements, then $\Q L$ has a simple component image which
is isomorphic to Zorn's matrix algebra over $\Q$.
\end{lemma}

\begin{proof} This is well known and follows from \cite[Corollary $VI.4.3$]{gjp} and 
\cite[Corollary $VI.4.8$]{gjp}.
\end{proof}
 
\begin{lemma}\label{tzlf} Let $L$ be a finite $RA$-loop. Then  $\U(\Z L)$ has the hyperbolic property
if and only if $\U_1(\Z L)$ is trivial, i.e., $\U_1(\Z L) = L$ .
\end{lemma}
\begin{proof} As we saw before, there exists a non-abelian finite
group $G$ such that $L=M(G,*,g_0)=G \displaystyle \dot{\cup}Gu$. Since $\U(\Z L)$ has the hyperbolic property, we have that $\Z^{2} {\not \hookrightarrow} \U(\Z G)$. As $G$ is finite,   \cite[Theorem $3.2$]{jpp} implies that 
 $G \in \{S_{3},D_{4},C_{3} \rtimes C_{4},C_{4} \rtimes C_{4} \} $ $\cup
\{M:M \textrm{ is a Hamiltonian $2$-group}\}.$  By  \cite[Theorem $3.1$]{bj}, $G^{'} \cong C_{2}$ and hence  $G \notin \{S_{3}, C_{3}\rtimes C_{4}\}$. We also have that $G \notin \{D_{4},C_{4} \rtimes
C_{4}\}$, since if this were the case then the algebra $\Q G$ would contain nilpotent
elements and thus, by Lemma \ref{sojj}, $\Q L$ would contain 
a copy of Zorn's matrix algebra, a contradiction.

Finally, if $G$ is a Hamiltonian $2$-group, then, by Proposition \ref{pgh},  $\U_1(\Z L)$ is trivial.
\end{proof}

 We can now characterize $RA$-loops whose integral loop ring has the hyperbolic property.

\begin{theorem}\label{kpm} Let $L$ be an $RA$-loop. The loop  $\U(\Z L)$ has the hyperbolic property if
 and only if the following conditions hold.
 \begin{enumerate} 
 \item The torsion subloop $T(L)$ is a Hamiltonian $2$-loop or an Abelian group of exponent dividing $4$ or $6$ or a Hamiltonian $2$-group.
 \item All subloops of $T(L)$ are normal in $L$.
 \item The Hirsh lenght of the center of $L$, $h(\z(L)) \leq 1$.  
\end{enumerate} 
 \end{theorem}

\begin{proof}
Assume that $\U(\Z L)$ has the hyperbolic property. As $\z (L)$ is abelian and has the hyperbolic property, it follows that $h(\z(L))\leq 1$.

If $h(\z(L))=0$, then $\z(L)$ is finite and, since, by \cite[Corollary $IV.2.3$]{gjp}, $[L:\z(L)]=8$ then also $L$ is finite. By Lemma \ref{tzlf}, $\U_1(\Z L)=L$ and thus $L$ and $T(L)$ are Hamiltonian, so $(1)$ and $(2)$ hold.

If $h(\z(L))=1$, assume, by the way of contradiction, that there exists $t \in T(L)$, such that $\langle t \rangle$ is not normal in $L$. Then, there exists $l \in L$, such that, $\theta=(1-t)l \hat{t} \neq 0$ where $\hat{t}=1+t+ \cdots + t^{o(t)-1}$. As  $\theta^2=0$, $u=1+\theta$ is a unit and $u^n=1+n \theta \notin L$ for all integers $n \neq 0$. Hence, $\theta$ is of infinite order and $\langle \theta \rangle \cap L=(1)$. So, as we are assuming that $h(\z(L))=1$, it follows that $\U(\Z L)$ contains a copy of $\Z^2$, a contradiction. Once again, we have that $(1)$ and $(2)$ hold.

Conversely, items $(1)$ and $(2)$ and  \cite[Proposition $XII.1.3$]{gjp} imply that  $\U_1(\Z L)=L[\U_1(\Z(T(L)))]=L( T(L))= L$.  Since $[L:\z (L)]=8$ and $\z (L)$ is hyperbolic, it follows that  $L$ has the hyperbolic property.

\end{proof}


%
%
%
%

\end{document}